\documentclass[a4paper,reqno]{amsart}

\usepackage{amsfonts}
\usepackage{amsmath}
\usepackage[colorlinks, citecolor={red}, linkcolor=.]{hyperref}
\usepackage{amssymb}
\usepackage{tikz}
\usepackage{color}
\usepackage[all]{xy}
\usepackage{soul}
\usepackage{enumerate}
\usepackage{bm}
\usepackage{bbm}
\usepackage[normalem]{ulem}
\usepackage{mathrsfs}
\usepackage{tikz-cd}
\usepackage{mathtools}
\usepackage{todonotes}
\usepackage{relsize}
\usepackage{mathdots}
\usepackage{setspace}

\usepackage[margin = 1.25in, centering]{geometry}

\allowdisplaybreaks
\numberwithin{equation}{section}
\newtheorem{theorem}{Theorem}[section]

\newtheorem{proposition}[theorem]{Proposition}
\newtheorem{lemma}[theorem]{Lemma}
\newtheorem{corollary}[theorem]{Corollary}

\theoremstyle{definition}
\newtheorem{definition}[theorem]{Definition}
\newtheorem{notation}[theorem]{Notation}
\newtheorem{construction}[theorem]{Construction}

\newtheorem{assumption}[theorem]{Assumption}
\theoremstyle{remark}
\newtheorem{remark}[theorem]{Remark}
\newtheorem{example}[theorem]{Example}

\DeclareMathOperator{\spec}{Spec}
\DeclareMathOperator{\spf}{Spf}
\DeclareMathOperator{\Loc}{Loc}
\DeclareMathOperator{\Cov}{Cov}
\DeclareMathOperator{\FEt}{FEt}
\DeclareMathOperator{\set}{\text{-}Set}
\DeclareMathOperator{\Set}{Set}
\DeclareMathOperator{\FSet}{FSet}
\DeclareMathOperator{\ev}{ev}
\DeclareMathOperator{\Gal}{Gal}
\DeclareMathOperator{\cont}{cont}

\DeclareMathOperator{\Hom}{Hom}
\DeclareMathOperator{\Inn}{Inn}
\DeclareMathOperator{\Shv}{Shv}
\DeclareMathOperator{\Aut}{Aut}
\DeclareMathOperator{\End}{End}
\DeclareMathOperator{\pr}{pr}
\DeclareMathOperator{\map}{map}

\renewcommand{\phi}{\varphi}

\newcommand{\efg}{\pi_1^{\mathrm{\acute{e}t}}}
\newcommand{\pefg}{\pi_1^{\mathrm{pro\acute{e}t}}}

\newcommand{\pet}{{\mathrm{pro\acute{e}t}}}

\newcommand{\XX}{\mathfrak{X}}
\newcommand{\Aa}{\mathcal{A}}
\newcommand{\Cc}{\mathcal{C}}
\newcommand{\Ff}{\mathcal{F}}
\newcommand{\Ll}{\mathcal{L}}
\newcommand{\Mm}{\mathcal{M}}

\newcommand{\Pp}{\mathcal{P}}

\newcommand{\IC}{\mathbb{C}}
\newcommand{\IF}{\mathbb{F}}
\newcommand{\IG}{\mathbb{G}}

\newcommand{\IP}{\mathbb{P}}
\newcommand{\IQ}{\mathbb{Q}}
\newcommand{\IR}{\mathbb{R}}

\newcommand{\IZ}{\mathbb{Z}}
\newcommand{\bigast}{\mathop{\Large \mathlarger{\mathlarger{\ast}}}}
\newcommand{\into}{\hookrightarrow}

\title[Quasi-isogenies of supersingular abelian surfaces]{Quasi-isogeny groups of supersingular abelian surfaces via pro-étale fundamental groups}

\author[T. van den Hove]{Thibaud van den Hove}

\address{Thibaud van den Hove: Fachbereich Mathematik, TU-Darmstadt, Germany}

\email{hove@mathematik.tu-darmstadt.de}


\begin{document}
	
	\maketitle
	
	\begin{abstract}
		We consider a \(J_b(\IQ_p)\)-torsor on the supersingular locus of the Siegel threefold constructed by Caraiani-Scholze, and show that it induces an isomorphism between a free group on a finite number of generators, and the group of self-quasi-isogenies of a supersingular abelian surface, respecting a principal polarization and a prime-to-\(p\) level structure. Along the way, we classify certain pro-étale torsors in terms of the pro-étale fundamental group, describe the category of geometric covers of non-normal schemes, and use this to compute pro-étale fundamental groups of curves.
	\end{abstract}
	
	\begin{spacing}{0.05}
		\tableofcontents
	\end{spacing}
	
	\section{Introduction}
	In \cite[Proposition 1.13]{generic}, Caraiani and Scholze construct a pro-étale \(J_b(\IQ_p)\)-torsor on each Newton stratum \(\mathscr{S}^b\) of certain PEL Shimura varieties at hyperspecial level, where \(J_b(\IQ_p)\) is the locally profinite group of self-quasi-isogenies of a certain \(p\)-divisible group, respecting extra structures. 
	For connected strata, these torsors induce a continuous morphism \[\pefg(\mathscr{S}^b,x)\to J_b(\IQ_p),\] where \(x\) is a geometric point of \(\mathscr{S}^b\), and \(\pefg\) is the pro-étale fundamental group of Bhatt and Scholze classifying geometric covers, as introduced in \cite[§7]{proetale}. 
	The main goal of this paper is to show that in the case of the basic stratum  of the Siegel threefold, this map is injective, and to determine its image. 
	More precisely, we compute the pro-étale fundamental group in this special case, and the image of the map to \(J_b(\IQ_p)\) consists exactly of those quasi-isogenies coming from a quasi-isogeny of abelian surfaces. 
	This gives the following main theorem (cf.~Remark \ref{pefg of ss locus}, and Propositions \ref{image of the map} and \ref{injectivity of the map}).
	
	\begin{theorem}\label{theorem 4}
		Let \(p\geq 3\) be prime, \(N\geq 3\) be prime to \(p\), and let \(V_N\) be the supersingular locus of the moduli space of principally polarized abelian surfaces over \(\overline{\IF_p}\) with level-\(N\) structure. Then the \(J_b(\IQ_p)\)-torsor \(\Pp_N\) on \(V_N\) considered above induces an isomorphism between a free group on a finite number of generators, and the group of self-quasi-isogenies of a supersingular abelian surface, respecting a principal polarization and a level-\(N\) structure.
	\end{theorem}
	
	Part of this theorem is an instance of the fact that for the basic stratum of general PEL Shimura varieties, the image of \(\pefg(\mathscr{S}^b,x)\to J_b(\IQ_p)\) is related to p-adic uniformization of Shimura varieties as in \cite{rapoportzink}, cf.~\cite[Remark 1.14]{generic}, although we do not make use of \(p\)-adic uniformization to determine this image. On the other hand, we do use it to get the corollary that the corresponding Rapoport-Zink space is simply connected. Indeed, this is the smallest cover of \(V_N\) trivializing \(\Pp_N\) (essentially by definition), so that this follows from injectivity of \(\pefg(V_N,x)\to J_b(\IQ_p)\).
	
	Before we mention other results we will use, let us sketch the strategy for Theorem \ref{theorem 4}. We consider families of principally polarized supersingular abelian surfaces with level structure over the projective line, constructed by Moret-Bailly, which give morphisms \(\IP^1_{\overline{\IF_p}}\to V_N\). By work of Katsura and Oort, these morphisms are jointly surjective, and each normalize an irreducible component of \(V_N\). This will allow us to compute \(\pefg(V_N,x)\), by Theorem \ref{theorem 2} below. Using that \(\IP^1_{\overline{\IF_p}}\) is simply connected, we then study the morphism \[\pefg(V_N,x)\to J_b(\IQ_p),\] and use arguments on the level of abelian surfaces to show injectivity and determine its image.
	
	To compute \(\pefg(V_N,x)\), we will prove more general results about the pro-étale fundamental group. To partially reduce the problem to determining étale fundamental groups, we show that all geometric covers of a scheme can be obtained by gluing finite étale covers of its normalization. More precisely, we prove Theorem \ref{theorem 1} below, cf.~Theorem \ref{geometric covers of seminormal schemes}. Here, we need some small topological assumption to ensure the pro-étale fundamental group exists and behaves well. Moreover, we assume the normal locus to be open, so that we can give the non-normal locus the structure of a closed subscheme. This holds e.g.~for quasi-excellent schemes \cite[Corollaire 6.13.5]{ega4.2}. 
	On the other hand, by topological invariance of the pro-étale fundamental group, we may assume reducedness without losing generality. For any scheme \(X\), we denote the category of geometric covers of \(X\) by \(\Cov_X\). Let us also note that a similar result, with different assumptions and proof, was obtained in \cite[Lemma 2.17]{lara2}: there, the authors allow more general proper covers than just the normalization map, but at the cost of assuming that the scheme is locally noetherian.
	
	\begin{theorem}\label{theorem 1}
		Let \(X\) be a locally topologically noetherian, connected, reduced scheme, whose normal locus is open. Denote by \(\pi\colon X^\nu\to X\) its normalization, \(Y\subseteq X\) its non-normal locus, and let \(Z=\pi^{-1}(Y)\), where both \(Y\) and \(Z\) have the reduced closed subscheme structure. Then base change induces an equivalence
		\[\Cov_X\cong \Cov_{X^\nu}\times_{\Cov_Z} \Cov_Y.\]
	\end{theorem}
	
	The 2-fibre product \(\Cov_{X^\nu}\times_{\Cov_Z} \Cov_Y\) can be viewed as a descent data category for the covering \(X^\nu\coprod Y\to X\); as \(Y\to X\) is a monomorphism, there is no need for cocycle conditions. 
	
	While Theorem \ref{theorem 1} above does not allow us to compute the pro-étale fundamental group in general, it suffices for the case of curves over separably closed fields, which is enough for our purposes. In that case, we get the formula below, cf.~Theorem \ref{pefg of curves over separably closed field}. Again, we note that a similar result was independently obtained in \cite[Theorem 2.27]{lara2}.
	
	\begin{theorem}\label{theorem 2}
		Let \(X\) be a connected curve over a separably closed field \(k\), with irreducible components \(X_0, \ldots, X_m\). Choose geometric basepoints \(x_i\) in the normal locus of \(X_i\) for each \(i\), and for each \(x\in X\), let \(b_x\) denote the number of branches at \(x\). Then there is an isomorphism
		\[\pefg(X,x_0) \cong F_{b(X)-m} *^N \bigast_{i=0,\ldots,m}^{\quad N} \efg(X_i^\nu,x_i),\]
		with \(F_{b(X)-m}\) a discrete free group on \(b(X)-m\) generators, and \(b(X)=\sum_{x\in X} (b_x-1)\).
	\end{theorem}
	
	Here, the notation \(*^N\) denotes the Noohi coproduct, but for this introduction, one can think of it as the usual topological coproduct, cf.~Remark \ref{coproducts of noohi groups}. 
	For curves over general fields, we can then use the \emph{fundamental exact sequence} from \cite{lara}. 
	Let us also mention another method from Lara to compute pro-étale fundamental groups, namely to use descent to deduce a general Van Kampen type theorem, cf.~\cite[Corollary 3.19]{lara}. Our method however, will show that elements of the free group \(F_{b(X)-m}\) can be thought of as loops in \(X\), recovering the topological intuition of fundamental groups. Moreover, this interpretation will be useful while proving Theorem \ref{theorem 4}.
	
	Finally, we need to know how the torsor \(\Pp_N\) induces a morphism of topological groups. This is an instance of the more general fact that the pro-étale fundamental group classifies pro-étale \(G\)-torsors for certain topological groups \(G\), similarly to the case of topological and étale fundamental groups. However, since pro-étale fundamental groups naturally live in the category of Noohi groups, this classification holds for all Noohi groups \(G\) (and in particular for the locally profinite groups); cf.~Theorem \ref{pefg classifies torsors}, which describes the whole groupoid of \(G\)-torsors.
	
	\begin{theorem}\label{theorem 3}
		Let \(X\) be a locally topologically noetherian connected scheme with geometric basepoint \(x\), and \(G\) a Noohi group. Then the isomorphism classes of pro-étale \(G\)-torsors on \(X\) are in bijection with \(\Hom_{\cont}(\pefg(X,x),G)/\Inn(G)\).
	\end{theorem}
	
	Note that we cannot expect the theorem to hold for groups that are not Noohi. This is similar to the case of étale fundamental groups, which are naturally profinite, but where the analogue of this theorem fails already for infinite discrete groups. As an example, the nodal curve over a separably closed field has a universal geometric cover, which is an étale \(\IZ\)-torsor. But since the étale fundamental group of this nodal curve is \(\widehat{\IZ}\), this would correspond to a non-trivial continuous homomorphism \(\widehat{\IZ}\to \IZ\), and such maps do not exist. This is another aspect with respect to which the pro-étale fundamental group behaves better than the étale version.
	
	\textbf{Outline.} In Section \ref{section recollections}, we recall the definitions and basic properties of the pro-étale fundamental group and Noohi groups, that we will need in the rest of the paper. In Section \ref{section geometric covers}, we explain how one can obtain geometric covers of non-normal schemes by gluing geometric covers of their normalizations, which leads to Theorem \ref{theorem 1}. Specializing to the case of curves, we use this in Section \ref{section computing} to prove Theorem \ref{theorem 2}, and to compute some examples of pro-étale fundamental groups of curves over non-separably closed fields. In Section \ref{section torsors}, we show Theorem \ref{theorem 3}, and in finally, in Section \ref{section siegel}, we work out the concrete example of a torsor on the basic stratum of the Siegel threefold, leading to Theorem \ref{theorem 4}.
	
	\textbf{Acknowledgements}
	First and foremost, I thank Johannes Anschütz for advising my Master's thesis at the University of Bonn, on which most of this work is based. I also thank Mingjia Zhang for her suggestion to look at the example of the Siegel threefold, and Peter Scholze for suggesting that the resulting map \(\pefg(V_N,x)\to J_b(\IQ_p)\) should be injective, and that this is equivalent to the corresponding Rapoport-Zink space being simply connected. Finally, I thank Torsten Wedhorn for helping me realize I could weaken some of the assumptions from my thesis, Timo Richarz, Jakob Stix and an anonymous referee for very helpful comments on earlier versions of this paper, and Marcin Lara for pointing out a mistake in an earlier draft.
	
	Part of this work was done while supported financially by the European Research Council (ERC) under the European Union’s Horizon 2020 research and innovation programme (grant agreement 101002592), and logistically by the Deutsche Forschungsgemeinschaft (DFG), through the TRR 326 \emph{Geometry and Arithmetic of Uniformized Structures} (project number 444845124).

	\section{The pro-étale fundamental group}\label{section recollections}
	In this section, let us recall the definition of the pro-étale fundamental group, and basic properties that we will need.
	
	\begin{definition}
		\begin{enumerate}
			\item A morphism \(f\colon Y\to X\) of schemes is \emph{weakly étale} if both \(f\) and the diagonal \(\Delta_f:Y\to Y\times_X Y\) are flat.
			\item The \emph{pro-étale site} \(X_\pet\) of \(X\) is defined as the category of weakly étale \(X\)-schemes, with fpqc covers.
		\end{enumerate}
		Let us denote the category of set-valued sheaves on \(X_\pet\) by \(\Shv(X_\pet)\), and by \(\Loc_X\) the full subcategory of locally constant sheaves.
	\end{definition}
	
	\begin{definition}\label{definition pefg}
		For a scheme \(X\) and a geometric basepoint \(x\) of \(X\), let \(\ev_x\colon \Loc_X\to \Set\) be the evaluation functor. We define the \emph{pro-étale fundamental group of \(X\) at the basepoint \(x\)} as \(\pefg(X,x):=\Aut(\ev_x)\). We endow it with the topology coming from the compact-open topology on each \(\Aut(\ev_x(\Ff))\), for \(\Ff\in \Loc_X\).
	\end{definition}
	
	For a topological group \(G\), we denote by \(G\set\) the category of discrete sets with a continuous \(G\)-action, which comes with a fibre functor \[F_G\colon G\set\to \Set,\] given by the forgetful functor. The main theorem about the pro-étale fundamental group is the following:
	\begin{theorem}[{\cite[Theorem 1.10]{proetale}}]
		If \(X\) is a locally topologically noetherian and connected scheme, then there is an equivalence \(\Loc_X\cong \pefg(X,x)\set\), compatible with fibre functors.
	\end{theorem}
	
	Note that some topological assumption on \(X\) is necessary, as there exist examples of connected schemes for which such an equivalence cannot hold, cf.~\cite[Example 7.3.12]{proetale}. It will also be useful to have a geometric interpretation of locally constant pro-étale sheaves, similarly to the fact that the finite locally constant étale sheaves are exactly those represented by finite étale covers.
	
	\begin{definition}
		A morphism \(f\colon Y\to X\) of schemes is a \emph{geometric cover} if it is étale and satisfies the valuative criterion for properness. We denote the category of geometric covers of \(X\) by \(\Cov_X\).
	\end{definition}
	
	\begin{proposition}[{\cite[Lemma 7.3.9]{proetale}}]
		If \(X\) is locally topologically noetherian, then \(\Loc_X=\Cov_X\), as subcategories of \(\Shv(X_\pet)\).
	\end{proposition}
	
	As for both the topological and the étale fundamental groups, the pro-étale fundamental group is independent of the choice of basepoint, up to (usually non-canonical) isomorphism. This is an immediate corollary of the proof of \cite[Lemma 7.4.1]{proetale}, but we state it explicitly below, as we will need it later on.
	
	\begin{lemma}\label{independence of choice of basepoint}
		Let \(X\) be a locally topologically noetherian connected scheme, and \(x_1\), \(x_2\) two geometric points of \(X\). Then there is an isomorphism \[\ev_{x_1}\cong \ev_{x_2}\] of functors \(\Loc_X\to \Set\).
	\end{lemma}
	
	Another property of locally topologically noetherian schemes, is that any étale morphism into such a scheme also has a locally topologically noetherian source. This was shown in \cite[Lemma 6.6.10]{proetale}, and it implies that any étale morphism into a locally topologically noetherian scheme is quasi-separated. We will use this fact without further mention in this paper.
	
	For the rest of this section, let us fix some locally topologically noetherian connected scheme \(X\), and a geometric basepoint \(x\) of \(X\). As the étale fundamental group \(\efg(X,x)\) is the automorphism group of the fibre functor \(\ev_x:\FEt_X\to \FSet\), there is a natural map \(\pefg(X,x)\to \efg(X,x)\). These two fundamental groups are closely related, and even the same for nice schemes:
	
	\begin{proposition}[{\cite[Lemma 7.4.3]{proetale}}]\label{efg vs pefg}
		The canonical continuous morphism \(\pefg(X,x)\to \efg(X,x)\) induces an isomorphism on profinite completions.
	\end{proposition}
	
	\begin{proposition}[{\cite[Lemma 7.4.10]{proetale}}]\label{pefg of normal schemes}
		If \(X\) is geometrically unibranch, the canonical morphism \(\pefg(X,x)\to \efg(X,x)\) is an isomorphism of topological groups.
	\end{proposition}
	
	\begin{example}[{\cite[Paragraph before Definition 1.9]{proetale}}]\label{pefg of nodal curve}
		To show that the pro-étale fundamental group is a strictly finer invariant than the étale fundamental group, consider the nodal curve \(X\) over a separably closed field \(k\), obtained by identifying the points \(0\) and \(\infty\) in \(\IP^1_k\). The connected geometric covers of \(X\) are then given by gluing copies of \(\IP^1_k\) along their points at \(0\) and \(\infty\). In particular there is a universal connected cover \(X_\infty\), for which \(\Aut(X_\infty/X)\cong \IZ\). So in this case, we have \[\pefg(X,x)\cong \IZ \ncong \widehat{\IZ} \cong \efg(X,x).\]
	\end{example}
	
	\begin{remark}
		In the example above, the pro-étale fundamental group is prodiscrete, and hence agrees with the enlarged fundamental group introduced in \cite[Exp.~X.6]{sga3.2}.
		However, in general the pro-étale fundamental group is a strictly finer invariant, cf.~\cite[Lemma 7.4.6 and Example 7.4.9]{proetale}.
		Another example is the pushout \(\IG_{m,\IC} \coprod_{\spec \IC} \IG_{m,\IC}\), where both maps \(\spec \IC\to \IG_{m,\IC}\) are given by the unit map. This already appears in \cite[Example 4.5]{lara}, but will also follow from Theorem \ref{pefg of curves over separably closed field} as \(\efg(\IG_{m,\IC},x) \cong \widehat{\IZ}\) for any basepoint \(x\).
	\end{remark}
	
	Another useful property of the pro-étale fundamental group, is that just like the étale fundamental group, it is a topological invariant:
	
	\begin{proposition}[Topological invariance]\label{topological invariance}
		A universal homeomorphism of locally topologically noetherian connected schemes induces an isomorphism of pro-étale fundamental groups.
	\end{proposition}
	
	This can easily be shown using topological invariance of the étale site, and the fact that separatedness and satisfying the existence part of the valuative criterion are topological properties. We refer to \cite[Proposition 2.17]{lara} for more details.
	
	As a last topic in this section, let us briefly mention the kind of groups that can appear. Throughout the rest of the paper, we will assume any topological group is Hausdorff. (Recall that this is equivalent to being a \(T_1\)-, or even a \(T_0\)-space. Indeed, in a \(T_0\)-group any point is closed, as translations and inversion are homeomorphisms. But then the diagonal is closed as the inverse image of the identity under the continuous map \(G\times G\to G:(x,y)\mapsto xy^{-1}\).)
	
	\begin{definition}
		Let \(G\) be a topological group, \(F_G\colon G\set\to \Set\) the forgetful functor, and consider the group \(\Aut(F_G)\), topologized similarly as in Definition \ref{definition pefg}. We say that \(G\) is \emph{Noohi} if the natural map \(G\to \Aut(F_G)\) is an isomorphism of topological groups.
	\end{definition}
	
	Since the equivalence \(\Loc_X\cong \pefg(X,x)\set\) is compatible with the fibre functors \(\ev_x\) and \(F_{\pefg(X,x)}\) to \(\Set\), any pro-étale fundamental group is Noohi.
	Below, a more intrinsic characterization of Noohi groups is given in terms of Raikov completeness and the Raikov completion \(G^*\) of a topological group \(G\), for which we refer to \cite[Section 3.6]{topgrps}. Note that the identity of any Noohi group has a basis of open neighbourhoods given by open subgroups, as this holds for groups of the form \(\Aut(S)\) with the compact-open topology, where \(S\) is a discrete set.
	
	\begin{proposition}[{\cite[Proposition 7.1.5]{proetale}}]\label{characterization of Noohi groups}
		If \(G\) is a topological group where the identity has a basis of open neighbourhoods given by open subgroups, then there is a natural isomorphism \(\Aut(F_G)\cong G^*\). In particular, \(G\) is Noohi if and only if \(G\) is Raikov complete.
	\end{proposition}
	
	\begin{example}\label{locally profinite is noohi}
		Using this characterization, we see that any locally profinite group is Noohi.
	\end{example}
	
	\begin{remark}[{\cite[Example 7.2.6]{proetale}}]\label{coproducts of noohi groups}
		This characterization shows that the product of two Noohi groups is Noohi, so that the category of Noohi groups admits products. And while it is not true that the coproduct of Noohi groups is always Noohi, the category of Noohi groups does admit coproducts: for two Noohi groups \(G\) and \(H\), it is given by \(\Aut(F_{G*H})\), where \(G*H\) is the coproduct of topological groups, and we denote it by \(G*^NH\). Using infinite Galois theory, cf.~\cite[Theorem 7.2.5]{proetale}, one sees that \((G*H)\set\cong (G*^NH)\set\), compatibly with both forgetful functors. For a different description of the Noohi coproduct as the Raikov completion of \(G*H\) endowed with a certain topology (\emph{not} the coproduct topology), we refer to \cite[Corollary 1.14]{lavanda}.
	\end{remark}

	\section{Geometric covers of non-normal schemes}\label{section geometric covers}
	
	In this section, we show how to reduce the computation of the pro-étale fundamental group to the case of a normal scheme, where it agrees with the étale fundamental group by Proposition \ref{pefg of normal schemes}. A natural candidate for such a normal scheme is the normalization \(X^\nu\) of \(X\). The idea is then similar to Example \ref{pefg of nodal curve}: we would like show that any geometric cover of \(X\) can be obtained by gluing finite étale covers of \(X^\nu\). However, as not every scheme can be obtained by such a gluing process, we will start by approximating \(X\) by a scheme that can be obtained that way, and then show how the pro-étale fundamental groups of \(X\) and this other scheme relate. For this section, let us fix a locally topologically noetherian connected scheme \(X\), which we may assume to be reduced by Proposition \ref{topological invariance}. Moreover, since the gluing procedure mentioned above will be formalized using pushouts, which are mostly only well-behaved when one morphism is a closed immersion, we will also assume that the normal locus of \(X\) is open in \(X\), so that the non-normal locus is closed. This holds for example for all quasi-excellent schemes, and hence for all schemes locally of finite type over a field; cf.~\cite[§6.13]{ega4.2} for this and other criteria.
	
	\begin{construction}
		Consider the normalization map \(\pi\colon X^\nu\to X\), let \(Y\subseteq X\) be the non-normal locus, and \(Z:=\pi^{-1}(Y)\). We view both \(Y\) and \(Z\) as closed subschemes of respectively \(X\) and \(X^{\nu}\), with the reduced subscheme structure. Let us consider the pushout \(X^{\nu} \coprod_{Z} Y\), which exists by \cite[Théorème 7.1]{ferrand}, as \(Z\to X^{\nu}\) is a closed immersion, and \(Z\to Y\) is integral. By the pushout property, there is a natural map \(\phi\colon X^{\nu} \coprod_{Z} Y\to X\). While this map is not an isomorphism in general, it is always a universal homeomorphism:
	\end{construction}
	
	\begin{lemma}\label{pushout is universal homeomorphism}
		The natural map \(\phi\colon X^{\nu} \coprod_{Z} Y\to X\) is a universal homeomorphism.
	\end{lemma}
	\begin{proof}
		By \cite[Théorème 7.1]{ferrand}, the underlying topological space of \(X^\nu\coprod_Z Y\) is the pushout of the respective topological spaces, \(Y\to X^\nu\coprod_Z Y\) is a closed immersion, and \(X^\nu\to X^\nu\coprod_Z Y\) is an isomorphism away from \(Y\). In particular, \(\phi\) is bijective and induces isomorphisms on residue fields, so by \cite[Corollaire 18.12.11]{ega4.4} we are left to show that \(\phi\) is integral. To show that \(\phi\) is affine, we note that the normalization map \(\pi\) is affine, and for any affine open \(U\subseteq X\), the preimage \(\phi^{-1}(U)\) is the image of \( U^{\nu}=\pi^{-1}(U)\subseteq X^\nu\) in \(X^{\nu} \coprod_{Z} Y\), as \(X^\nu\to X^{\nu} \coprod_{Z} Y\) is surjective. But this image is just \(U^\nu\coprod_{Z_U} Y_U\), where \(Y_U:= Y\times_X U\) and \(Z_U:=Z\times_X U\) are both affine. And as the pushout of affine schemes is again affine (as the spectrum of the fibre product on the level of rings), \(\phi\) is also affine. So we may assume that \(X=\spec A\), and similarly \(X^{\nu} \coprod_{Z} Y=\spec A'\) and \(X^\nu=\spec A^\nu\), where we know that \(A\to A^\nu\) is integral. Now, \(A'\) is just the fibre product of \(A^\nu\) along an injection of rings, as \(Z\to Y\) is a dominant morphism of reduced schemes. And because fibre product of rings preserves injectivity, we see that \(A'\to A^\nu\) is injective. Since \(A^\nu\) is integral over \(A\), the same then holds for \(A'\), so that \(\phi\) is integral, and hence a universal homeomorphism.
	\end{proof}
	
	In particular, by Proposition \ref{topological invariance}, the pro-étale fundamental groups of \(X\) and \(X^\nu\coprod_Z Y\) coincide.
	
	\begin{remark}
		The reason we introduced the pushout \(X^{\nu} \coprod_{Z} Y\) is that, although it will still be non-normal in general, its singularities are better behaved than those of \(X\). And since the natural map \(X^\nu\to X^\nu\coprod_Z Y\) is surjective and birational, it is a normalization morphism, so that taking the pushout \(X^\nu\coprod_Z Y\) is a weaker process than normalizing. In fact, under certain assumptions on \(X\), we can show that \(X^\nu\coprod_Z Y\) is exactly the seminormalization of \(X\), cf.~Remark \ref{pushout is seminormalization for curves}.
	\end{remark}

	To simplify the notation, we will from now on write \(\hat{X}\) instead of \(X^\nu \coprod_Z Y\). Let us now describe the category of geometric covers of \(\hat{X}\), similarly to \cite[Tag 0ECL]{stacks}.
	
	\begin{theorem}\label{geometric covers of seminormal schemes}
		Base change induces an equivalence
		\[\Cov_{\hat{X}}\cong \Cov_{X^{\nu}}\times_{\Cov_Z} \Cov_Y.\]
	\end{theorem}
	Recall from the introduction that one can view this 2-fibre product as a descent category for the covering \(X^\nu\coprod Y\to X\), and that no cocycle condition is needed as \(Y\to X\) is a monomorphism.
	\begin{proof}
		First, we note that \(\Cov_{X^{\nu}}\times_{\Cov_Z} \Cov_Y\) can be identified with the category \(\Cc\) of diagrams of the form
		\[\begin{tikzcd}
			X' \arrow[d, "f"] & Z' \arrow[d, "g"] \arrow[r,"j'"] \arrow[l, "i'"'] & Y' \arrow[d,"h"]\\
			X^\nu & Z \arrow[r, "j"'] \arrow[l, "i"] & Y,
		\end{tikzcd}\]
		where the vertical morphisms are geometric covers and the two squares cartesian, with the obvious morphisms between them. So it is enough to find an equivalence \(\Cov_{\hat{X}}\cong \Cc\). For a geometric cover \(W\to \hat{X}\), let \(X':=X^\nu \times_{\hat{X}} W\), and similarly \(Y':= Y\times_{\hat{X}} W\) and \(Z':= Z\times_{\hat{X}} W\). With the obvious morphisms, this gives a diagram in \(\Cc\), since being a geometric cover is stable under base change, and because we have \[X' \times_{X^\nu} Z = W\times_{\hat{X}} X^\nu \times_{X^\nu} Z \cong W\times_{\hat{X}} Z = Z',\] and similarly for \(Y'\).
		
		Conversely, if we have a such a diagram in \(\Cc\), we want to show the pushout \(X'\coprod_{Z'} Y'\) exists, and is a geometric cover of \(X^\nu \coprod_Z Y=\hat{X}\). We will show this pushout exists under the additional assumption that \(\pi\colon X^\nu\to X\) is finite, and refer to the first paragraph of the proof of \cite[Tag 0ECK]{stacks} for the general case. By \cite[Théorème 7.1]{ferrand}, we have to show that for any point \(y'\in Y'\), there is an open affine \(U'\subseteq X'\) such that \(j'^{-1}(y')\) is contained in \(U'\). If the three vertical maps \(f\), \(g\) and \(h\) in our diagram are the identity, this holds because the normalization map is affine. Otherwise, note that as \(X^\nu\) is normal, \(X'\) is the disjoint union of finite étale covers of \(X^\nu\). And by our assumption that the normalization map is finite, \(Z'\to Y'\) is finite as well, so that \(y'\) only has finitely many preimages in \(Z'\), which are contained in the disjoint union \(X''\) of finitely many components of \(X'\), each of which is finite étale over \(X^\nu\). So we can consider the image \(y=h(y')\), an affine open \(U\subseteq X^\nu\) containing \(j^{-1}(y)\), and then the inverse image \(f_{\mid X''}^{-1}(U)\subseteq X''\) is the open affine we are looking for.
		
		Using \cite[Tag 08KQ]{stacks} (resp. \cite[Tag 0ECK]{stacks}), we find that \(X'\coprod_{Z'} Y'\to X^\nu\coprod_Z Y\) is étale (resp. separated), so we are left to show the existence part of the valuative criterion. Fix a valuation ring \(V\) with fraction field \(K\), and consider a commutative diagram as follows:
		\[\begin{tikzcd}
			\spec K \arrow[d] \arrow[r, dotted] \arrow[rr, bend left] & X' \arrow[d] \arrow[r] & X' \coprod_{Z'} Y'\arrow[d] & Y' \arrow[l] \arrow[d]\\
			\spec V \arrow[r, dotted] \arrow[rr, bend right] & X^\nu \arrow[r]& \hat{X} & Y\arrow[l],
		\end{tikzcd}\]
		considering only the solid arrows at first. If the image of \(\spec K\) in \(X'\coprod_{Z'} Y'\) is in the normal locus, we can uniquely lift this map to a map \(\spec K\to X'\), and compose it to get a map \(\spec K\to X^\nu\). Since \(X^\nu\to \hat{X}\) is a normalization map, it is integral, and hence universally closed. In particular, it satisifies the existence part of the valuative criterion, so we can lift \(\spec V\to \hat{X}\) to a map \(\spec V\to X^\nu\), and we get two dotted arrows as in the diagram. Since \(X'\to X^\nu\) is a geometric cover by assumption, we also get a lift \(\spec V\to X'\), which we can compose with \(X'\to X'\coprod_{Z'} Y'\) to get the desired lift. On the other hand, if the image of \(\spec K\to X'\coprod_{Z'} Y'\) is in the non-normal locus, we can uniquely lift this map to \(\spec K\to Y'\). And since \(Y\to \hat{X}\) is a closed immersion, the same argument as above gives a lift \(\spec V\to X'\coprod_{Z'} Y'\). We conclude that \(X'\coprod_{Z'} Y'\to \hat{X}\) is a geometric cover.
		
		Both mappings on objects of \(\Cc\) and \(\Cov_{\hat{X}}\) can be upgraded to functors, and we want to show they are mutual quasi-inverses. We will only check this on objects, as once we know this the case of morphisms is easy. If we have a diagram \(D\) in \(\Cc\), take the associated geometric cover, and then the diagram in \(\Cc\) obtained by fibre products, this new diagram will be isomorphic to \(D\) by \cite[Tag 07RU]{stacks} and cartesianness of the squares of \(D\).
		
		Conversely, let \(W\) be a geometric cover of \(\hat{X}\), and construct \(X'\), \(Y'\) and \(Z'\) as above. There is a natural map \(X'\coprod_{Z'} Y'\to W\), and we can check locally that it is an isomorphism. To do this, let \(A\to C \leftarrow B\) be a fibre diagram of rings, and let \(M\) be a flat ring over \(A\times_C B\). Then we have an exact sequence \(0\to A\times_C B \to A\oplus B \to C\) of modules. Tensoring this with \(M\) over \(A\times_C B\) gives another exact sequence, which realizes \(M\) as the fibre product of \(M\otimes_{A\times_C B} A\) and \(M\otimes_{A\times_C B} B\) over \(M\otimes_{A\times_C B} C\). Since this is just the affine version of the construction above, we are done.
	\end{proof}

	\section{Computing the pro-étale fundamental group}\label{section computing}
	
	Using Theorem \ref{geometric covers of seminormal schemes} to compute \(\pefg(X,x)\) can still be difficult if \(Y\) and \(Z\) have complicated geometric covers. This happens to a lesser extent if \(Y\) and \(Z\) consist of points, such as when \(X\) is a curve, and we will see that if \(X\) is moreover defined over a separably closed field, we will actually be able to find a formula for \(\pefg(X,x)\). Note that we will not need any separatedness or irreducibility assumptions on \(X\), and by topological invariance we do not lose any generality by assuming reducedness.
	
	\begin{assumption}
		For this section, assume \(X\) is connected, reduced, of finite type over a field \(k\), and that its non-normal locus is zero-dimensional, i.e., consists of finitely many closed points.
	\end{assumption}
	
	\begin{remark}\label{pushout is seminormalization for curves}
		Under these assumptions, let us show that \(X^{\nu}\coprod_Z Y\) is exactly the seminormalization of \(X\). Recall that this seminormalization is the initial scheme \(X^s\) with a universal homeomorphism to \(X\) that induces isomorphisms on residue fields. So we immediately get a morphism \(\alpha\colon X^s\to X^{\nu}\coprod_Z Y\). On the other hand, because of our assumption on \(X\), we also have a natural morphism \(Y\to X^{s}\) such that the compositions \(Z\to Y\to X^s\) and \(Z\to X^{\nu} \to X^s\) agree (recall that the normalization always factors through the seminormalization, so that we also have a map \(X^{\nu}\to X^s\)). The pushout property then gives a map \(\beta\colon X^{\nu}\coprod_Z Y\to X^s\). Since maps obtained by universal properties are unique, the compositions \(\alpha\circ \beta\) and \(\beta\circ \alpha\) must be the identity maps, so that \(\alpha\) and \(\beta\) are mutual inverses.
	\end{remark}
	
	Let us recall the notion of \emph{(geometric) branches} of a point of a scheme:
	
	\begin{definition}
		Let \(T\) be a scheme which locally has finitely many irreducible components, and \(\pi\colon T^\nu\to T\) its normalization. For a point \(t\in T\), one defines:
		\begin{enumerate}
			\item The number of \emph{branches} of \(T\) at \(t\) is the number of inverse images of \(t\) in \(T^\nu\).
			\item The number of \emph{geometric branches} of \(T\) at \(t\) is \(\sum_{t^\nu\in \pi^{-1}(t)} [k(t^\nu):k(t)]_{\text{sep}}\).
		\end{enumerate}
	\end{definition}
	
	Clearly, for schemes of finite type over a separably closed field, the two definitions agree. Note also that if we denote the number of geometric branches at \(x\in X\) by \(b_x\), then by the assumption on our scheme \(X\), the number of points \(x\in X\) for which \(b_x>1\) is finite. In particular, we can give sense to the infinite sum \(\sum_{x\in X} (b_x-1)\). Finally, note that since \(X^\nu\) is the normalization of both \(X\) and \(\hat{X}\), and because \(\hat{X}\to X\) is a homeomorphism which induces isomorphisms on residue fields, the number of (geometric) branches of \(X\) at some point agrees with the number of (geometric) branches at the corresponding point of \(\hat{X}\). 
	
	These numbers of branches allow us to find a nice formula for the pro-étale fundamental group of \(X\), which we prove in the following theorem. This generalizes a result obtained in \cite[Proposition 1.17]{lavanda} for the case of projective normal crossing curves. It is also similar to \cite[Theorem 2.27]{lara2}; we note that both results were obtained independently. 
	The picture to keep in mind is that of the fundamental group of a graph of groups.
	Indeed, to \(X\) we can attach a graph of groups, where the vertices correspond to the irreducible components of \(X\), equipped with the étale fundamental group of their normalization, and points on \(X\) with multiple branches give rise to edges, equipped with the trivial group.
	Then the formula below is similar to \cite[Example 1 on p.~43]{serretrees}, adapted to the setting of Noohi groups.
	Similar computations have also appeared in \cite[Exp.~XI, Corollaire 5.4]{sga1} (and more generally in \cite{stix}) for the étale fundamental group, and a specific example can be found in \cite[Example 3.25]{lara}.
	
	Recall that \(*^N\) denotes the coproduct of Noohi groups, as in Remark \ref{coproducts of noohi groups}.

	\begin{theorem}\label{pefg of curves over separably closed field}
		Assume \(k\) is separably closed, and let \(X_0,\ldots,X_m\) be the irreducible components of \(X\). Choose geometric basepoints \(x_i\) of \(X_i\) with closed image in the normal locus of \(X\), for each \(i\). Then there is an isomorphism 
		\[\pefg(X,x_0) \cong F_{b(X)-m} *^N \bigast_{i=0,\ldots,m}^{\quad N} \efg(X_i^\nu,x_i),\]
		with \(F_{b(X)-m}\) a discrete free group on \(b(X)-m\) generators, and \(b(X)=\sum_{x\in X} (b_x-1)\).
	\end{theorem}
	
	\begin{proof}
		Since we are working over a separably closed field, we can simplify our notation and not distinguish between a geometric point and its closed image. Similarly, we will not distinguish a point of a scheme from its image under a closed immersion.
		
		By Proposition \ref{topological invariance} and Lemma \ref{pushout is universal homeomorphism}, we can replace \(X\) by \(\hat{X}\). By Remark \ref{coproducts of noohi groups}, it is enough to show \[\Cov_{\hat{X}} \cong \big(F_{b(X)-m}* \bigast_{i=0,\ldots,m} \efg(X_i^\nu,x_i)\big)\set.\] By Theorem \ref{geometric covers of seminormal schemes}, we can instead show that \(\big(F_{b(X)-m}* \bigast_{i=0,\ldots,m} \efg(X_i^\nu,x_i)\big)\set\) is equivalent to \(\Cc\), where \(\Cc\) is the category of diagrams appearing in the proof of the aforementioned theorem. To do this, we make the following choices, using Lemma \ref{independence of choice of basepoint}:
		\begin{itemize}
			\item For each \(z\in Z\subseteq X^\nu\) (i.e., for each \(z\in X^\nu\) whose image in \(\hat{X}\) is not normal), fix a natural isomorphism \(F_z\colon \ev_{x_i}\xrightarrow{\cong} \ev_z\) of functors \(\Cov_{X_i^\nu}\to \Set\), where \(i\) is the index such that \(z\in X_i^\nu\). By base change, this gives a natural isomorphism \(\overline{F}_z\colon \ev_{x_i}\xrightarrow{\cong} \ev_{\overline{z}}\) of functors \(\Cov_{\hat{X}}\to \Set\), with \(\overline{z}\) the image of \(z\) in \(\hat{X}\).
			\item Rearranging the indices of the irreducible components of \(\hat{X}\), we may assume that for each \(i=1,\ldots,m\), there is some \(j<i\) such that \(\hat{X}_i\cap \hat{X}_j\neq \varnothing\). For any index \(i\), fix some point \(y_i\) in such an intersection, and some points \(y_i^i\in X_i^\nu\) and \(y_i^j\in X_j^\nu\) lying above it. In particular, we get natural isomorphisms \(F_i\colon \ev_{x_0}\cong \ev_{x_i}\) of functors \(\Cov_{\hat{X}}\to \Set\), by composing the isomorphisms
			\[\ev_{x_0} \to \ldots \to \ev_{x_j} \xrightarrow{\overline{F}_{y_i^j}} \ev_{y_i} \xrightarrow{\overline{F}^{-1}_{y_i^i}} \ev_{x_i}.\]
			\item For each \(x\in \hat{X}\) with \(b_x>1\) that is not one of the \(y_i\)'s, choose some \(x^\nu\in X^\nu\) over \(x\). And to simplify the notation later on, let us denote \(y_i^j\) by \(y_i^\nu\), where \(j<i\).
			\item Finally, we fix a set of free generators \(T\) of \(F_{b(X)-m}\), and an identification of \(T\) with the set of \(z\in Z\) which are not one of the choices for \(y_i^i\), \(y_i^j\) or \(x^\nu\) made above. (Note that the number of such \(z\in Z\) is exactly \(b(X)-m\).)
		\end{itemize}
		Now, supppose we have a diagram
		\[D=\begin{tikzcd}[row sep=small]
			X' \arrow[d] & Z' \arrow[d] \arrow[r] \arrow[l] & Y' \arrow[d]\\
			X^\nu & Z \arrow[r] \arrow[l] & Y,
		\end{tikzcd}\]
		in \(\Cc\), and denote the preimage of \(X_i^\nu\) in \(X'\) by \(X'_i\). Since \(x_0\) lies in the normal locus of \(\hat{X}\), the fibre functor \(F_{\Cc}\) of \(\Cc\) (corresponding to the fibre functor \(\ev_{x_0}\) of \(\Cov_{\hat{X}}\) under the equivalence \(\Cc\cong \Cov_{\hat{X}}\)) sends this diagram to the set of points of \(X'\) over \(x_0\). Considering only \(X'_0\), we see that \(F_{\Cc}\) admits an obvious action by \[\pefg(X_0^\nu,x_0)\cong \efg(X_0^\nu,x_0).\] Similarly, using the fixed isomorphisms \(F_i\colon \ev_{x_0}\cong \ev_{x_i}\), we obtain actions of \[\pefg(X_i^\nu,x_i)\cong \efg(X_i^\nu,x_i)\] on \(F_{\Cc}\). To define an action of \(F_{b(X)-m}\) as well, it is enough to specify \(b(X)-m\) automorphisms of the functor \(F_{\Cc}\). For each \(z\in T\) in \(X_i^{\nu}\) with image \(w\in \hat{X_i}\), we define the following automorphism \(t_z\): 
		for a point \(p\in F_{\Cc}(D)\), consider \(F_z(X'_i)\circ F_i(X')(p)\). This is a point of \(Z'\) lying over \(z\in Z\). Let \(q\) be the unique point of \(Z'\) which gets mapped to the same point in \(Y'\), but also gets mapped to \(w^\nu\in Z\). Then we define \(t_z(p)=F_j(X')^{-1} \circ F_{w^\nu}(X'_j)^{-1}(q)\), where \(j\) is the index such that \(w^\nu\in X_j^\nu\).
		This defines a natural automorphism of \(F_{\Cc}\), as \(F_i\), \(F_j\), \(F_z\) and \(F_{w^\nu}\) are natural isomorphisms, and because the morphisms in \(\Cc\) are induced by morphisms of \(\Cov_{X^\nu}\), \(\Cov_Y\) and \(\Cov_Z\). (Note the slight abuse of notation we use by composing \(F_{z}\) and \(F_i\), but since \(x_i\) lies in the normal locus of \(\hat{X}\), this does not lead to any problems.)
		
		Conversely, let \(S\) be a set with a continuous \(\big(F_{b(X)-m}* \bigast_{i=0,\ldots,m} \efg(X_i^\nu,x_i)\big)\)-action. Consider, for each \(i\), the set \(S\) with the restricted \(\efg(X_i^\nu,x_i)\cong \pefg(X_i^\nu,x_i)\)-action. By the equivalence \(\pefg(X_i^\nu,x_i)\set \cong \Cov_{X_i^\nu}\) we get a geometric cover \(X'_i\) of \(X_i^\nu\), and we define \(X':=\bigsqcup_{i=0}^m X'_i\) and \(Z':=X'\times_{X^n} Z\). Since we are working over a separably closed field, both \(Z\) and \(Z'\) are disjoint unions of copies of \(\spec k\), and we have \(Z'\cong \coprod_S Z\). So to get a diagram in \(\Cc\), we have to take \(Y'=\coprod_S Y\to Y\), and we are left to determine the map \(Z'\to Y'\), i.e., how points of \(Z'\subseteq X'\) are glued together.
		
		We start by gluing \(X'_i\)'s and \(X'_j\)'s together, for \(i\neq j\). We do this inductively on \(i=1,\ldots,m\). For \(i=1\), fix an identification of \(S\) with both the points of \(X'_0\) lying over \(x_0\) and the points of \(X'_1\) lying over \(x_1\), compatibly with the actions of \(\efg(X_0^\nu,x_0)\) and \(\efg(X_1^\nu,x_1)\) respectively. We then map, for each \(s\in S\), the points \(F_{y_1^0}(X'_0)(s)\) and \(F_{y_1^1}(X'_1)(s)\) to the same point in \(Y'\) (which maps to the right point in \(Y\)), using the identification we just made. This determines the images of points lying over \(y_1^0\) and \(y_1^1\), and we can repeat this process for \(i>1\).
		
		For each \(x\in \hat{X}\) with \(b_x>1\) that was not of the form \(y_i\), we had fixed some \(x^\nu\in X^\nu\) lying over \(x\). Points of \(Z'\) lying over these \(x^\nu\) do not need to be glued to the points considered in the previous paragraph, so one can simply choose their image in \(Y'\) (with correct image in \(Y\)), such that they all map to distinct points.
		
		Finally, consider some \(z\in T\) with \(z\in X_l^\nu\). Then \(z\) is not of the form \(y_i^i\), \(y_i^j\) or \(x^\nu\). But any point \(z'\) lying over \(z\) must be glued together with exactly one point lying over some \(x^\nu\) or some \(y_i^j=y_i^\nu\) (let us denote both by \(x^\nu\in X_j^\nu\)), and the images in \(Y'\) of these points are already determined. To determine to which point \(z'\) must be glued, consider the automorphism \(t_z\) on \(S\) associated to \(z\in T\subseteq F_{b(X)-m}\). Looking at the way that we constructed an automorphism associated to \(z\) before, we see that \(z'\) must be glued to the point \(F_{x^\nu}(X'_j) \circ F_j(X') \circ t_z \circ F_l(X')^{-1} \circ F_z(X'_l)^{-1}(z')\).
		
		The constructions above can naturally be upgraded to functors, which are readily seen to be mutual quasi-inverses. This concludes the proof.
	\end{proof}
	
	\begin{remark}
		In a certain sense, the elements of \(F_{b(X)}\subseteq \pefg(X,x_0)\) correspond to loops in \(X\), which recovers some of the topological intuition for fundamental groups.
	\end{remark}
	
	\begin{remark}
		Even if \(X\) is irreducible, it is relatively difficult to determine \(\efg(X^\nu,x)*^NF_{b(X)}\) completely: if \(\efg(X^\nu,x)\) is infinite profinite, then \(\efg(X^\nu,x)*F_{b(X)}\) will have some non-discrete topology, so it might not be Noohi. However, we do know that the natural map \[\efg(X^\nu,x)*F_{b(X)}\to \efg(X^\nu,x)*^N F_{b(X)}\] is injective. 
		Indeed, since the latter is the automorphism group of \(\Cov_{\hat{X}}\), we have to show that any non-trivial element \(\psi\in \efg(X^\nu,x)*F_{b(X)}\) acts non-trivially on the image under \(\ev_x\) of some geometric cover of \(\hat{X}\). To construct such a cover, consider the (unique) reduced decomposition of \(\psi\) as a product of elements in \(\efg(X^\nu,x)\) and in \(F_{b(X)}\), and let \(Y\to X^\nu\) be a connected finite étale cover on which some element of \(\efg(X^\nu,x)\) appearing in the decomposition of \(\psi\) acts non-trivially. (This is always possible unless \(\psi\) is in the image of \[F_{b(X)}\to \efg(X^\nu,x)*F_{b(X)}.\] In that case, taking \(Y=X^\nu\) will suffice, as we assumed \(\psi\) was non-trivial.) Let \(Y_\infty\) be the universal connected geometric cover of \(\hat{X}\) whose irreducible components are \(Y\). It is then clear that \(\psi\) acts non-trivially on \(\ev_x(Y_\infty)\).
	\end{remark}
	
	\begin{remark}
		Simarly as for the étale fundamental group, there is a fundamental exact sequence for the pro-étale fundamental group, cf.~\cite[Theorem 4.14]{lara}: if \(Y\) is a geometrically connected scheme of finite type over \(k\), then there is a short exact sequence
		\[1\to \pefg(Y_{k^{\text{sep}}})\to \pefg(Y)\to \Gal(k^{\text{sep}}/k)\to 1\]
		of topological groups. In particular, using this together with our Theorem \ref{pefg of curves over separably closed field}, we can determine the pro-étale fundamental group of a geometrically connected curve, up to some extension problem. For simple cases however, we can fully determine the pro-étale fundamental group, as the following examples show.
	\end{remark}
	
	\begin{example}\cite[Example 3.24]{lara}\label{example projective line over any field}
		Consider a curve \(X\) obtained by gluing the \(k\)-rational points 0 and \(\infty\) of the projective line \(\IP^1_k\) together. (This is similar as in Example \ref{pefg of nodal curve}, except that we allow \(k\) to be any field.) For any étale field extension \(k'/k\), we get a geometric cover of \(X\) by gluing copies of \(\IP^1_{k'}\), either a finite or an infinite number, by using Theorem \ref{geometric covers of seminormal schemes}. To show that, up to isomorphism, we can get all the connected geometric covers of \(X\) this way, we have to show we cannot glue some \(\IP^1_{k_1}\) to some \(\IP^1_{k_2}\) if \(k_1\ncong k_2\). Indeed, there were such a geometric cover, it would correspond to a diagram
		\[\begin{tikzcd}
			X' \arrow[d] & Z' \arrow[d] \arrow[r] \arrow[l] & Y' \arrow[d]\\
			\IP^1_k & \spec k\bigsqcup \spec k \arrow[r] \arrow[l] & \spec k,
		\end{tikzcd}\]
		where \(Z'\) contains \(\spec k_1\) and \(\spec k_2\) as clopen subschemes, which are mapped to the same point of \(Y'\). But the right square of this diagram must be a fibre square, which shows this can only happen if \(k_1\cong k_2\). So we have found all connected geometric covers of \(X\). Let \(x\) be a geometric point of \(X\). For the pro-étale fundamental group, note that it is generated by \(\IZ\) and \(\Gal(\overline{k}/k)\), where \(\IZ\) acts by translation, and \(\Gal(\overline{k}/k)\) by the usual Galois action on the irreducible components. Indeed, we are only gluing points of \(\IP^1_k\), and \(\efg(\IP^1_k,x)\cong \Gal(\overline{k}/k)\). Moreover, it is readily seen that these two actions commute with each other, so that \(\pefg(X,x)\cong \IZ\times \Gal(\overline{k}/k)\), with the product topology. 
	\end{example}
	
	So far, we only considered curves whose normalization map induced isomorphisms on residue fields. Let us see what can happen if this is not the case.
	
	\begin{example}
		We want to consider a scheme \(X\) which looks like the complex projective line, except that the residue field at \(\infty\) is \(\IR\) instead of \(\IC\). Formally, we can construct this as \(\IP^1_{\IC} \coprod_{\spec \IC} \spec \IR\), where \(\spec \IC\into \IP^1_{\IC}\) is the inclusion of \(\infty\). Using Theorem \ref{geometric covers of seminormal schemes}, we can see that, up to isomorphism, \(X\) has only two connected geometric covers: the identity, and a scheme \(Y\) gotten by gluing two copies of \(\IP^1_{\IC}\) together along their respective points at infinity. This corresponds to the diagram
		\[\begin{tikzcd}
			\IP^1_{\IC} \bigsqcup \IP^1_{\IC} \arrow[d] & \spec \IC \bigsqcup \spec \IC \arrow[d] \arrow[r] \arrow[l] & \spec \IC \arrow[d]\\
			\IP^1_{\IC} & \spec \IC \arrow[r] \arrow[l] & \spec \IR.
		\end{tikzcd}\]
		Indeed, since \(\IP^1_{\IC}\) has no nontrivial geometric covers, we see that any such diagram where the upper right scheme is not connected will give rise the a geometric cover that is not connected. Since \(Y\to X\) is a cover of degree 2, we see that for any geometric point \(x\) of \(X\), we have \(\pefg(X,x)\cong \IZ/2\IZ\), with the discrete topology.
	\end{example}
	
	\begin{example}
		Similarly as in the example above, consider a scheme \(X\) which looks like the complex projective line, but which has residue field \(\IR\) at two points: 0 and \(\infty\). 
		While it is possible to determine all the connected geometric covers, let us just mention that \(X\) has a universal geometric cover \(Y\). It is obtained by taking countably many copies of \(\IP^1_{\IC}\), and gluing them together by repeatedly identifying two points lying over \(0\in X\), and two points over \(\infty\in X\) (this is in contrast to Example \ref{example projective line over any field}, where we identified points lying over \(0\in \IP^1_k\) with points over \(\infty\in \IP^1_k\)).
		So the pro-étale fundamental group of \(X\) is discrete, and isomorphic to the automorphism group of \(Y\) over \(X\). This automorphism group is generated by translations, and the Galois action of \(\IC\) over \(\IR\). Since this Galois action essentially reverses the order of the copies of \(\IP^1_{\IC}\),  we see that \(\pefg(X,x)\cong \IZ \rtimes \IZ/2\IZ\) for any geometric point \(x\) of \(X\), where the action of \(\IZ/2\IZ\) is the unique non-trivial action.
	\end{example}
	
	\begin{remark}
		By \cite[Tags 0C1S and 0C39]{stacks}, an integral scheme is (geometrically) unibranch at a point if and only if the number of (geometric) branches at that point is 1. So the example above gives a scheme which is unibranch, but whose étale and pro-étale fundamental group are not isomorphic. This shows that the geometrically unibranch assumption from Proposition \ref{pefg of normal schemes} cannot be weakened to only unibranch.
	\end{remark}

	\section{Pro-étale torsors}\label{section torsors}
	
	Although the pro-étale fundamental group of a scheme \(X\) classifies the geometric covers of \(X\), we already had to determine the whole category of geometric covers of \(X\) to compute its pro-étale fundamental group. However, the pro-étale fundamental group can also be used to classify more general torsors, as we will show. Throughout this section, \(X\) again denotes a locally topologically noetherian connected scheme, and \(x\) a geometric basepoint of \(X\).
	
	\begin{notation}
		For any topological space \(T\), denote by \(\Ff_T\) the presheaf on \(X_\pet\) given by \(U\mapsto \cont(U,T)\). This is a sheaf by \cite[Lemma 4.2.12]{proetale}, and if \(T\) is a topological group or monoid, this gives \(\Ff_G\) the structure of a pro-étale sheaf of groups or monoids. 
		For a topological group \(G\), we denote by \(B\Ff_G(X_\pet)\) the groupoid of (right) \(\Ff_G\)-torsors on \(X_\pet\), with equivariant morphisms.
	\end{notation}
	
	To construct such torsors later in this section, we will need to consider the limit \(\varprojlim_U G/U\) of topological spaces, where \(U\) ranges over the set of open subgroups of a Noohi group \(G\). As these form a basis of open neighbourhoods of \(1\in G\), there is a natural injective map \[G\to \varprojlim_U G/U,\] but it is not surjective in general\footnote{This failure of surjectivity and the counterexample were pointed out to me by Marcin Lara.}: let \(S\) be a discrete set, and \(G=\Aut(S)\) with the compact-open topology. Then \(G\) is a Noohi group by \cite[Example 7.1.2]{proetale}, and a basis of open neighbourhoods of \(1\in \Aut(S)\) is given by the pointwise stabilizers \(U_F\) of finite subsets \(F\subseteq S\). For each such \(F\), there is a natural injection \(G/U_F\to \map(F,S)\), with image the injective maps \(F\to S\). Passing to the limit, we get a continuous injection \[\varprojlim_F G/U_F \to \varprojlim_F \map(F,S) = \map(S,S),\] and its image are again exactly the injective maps. In particular, if \(S\) is infinite, the map \(G=\Aut(S)\to \varprojlim_F G/U_F\) is not surjective.
	
	What does hold, is that there is a natural homeomorphism \(\varprojlim_U G/U \cong \End(F_G)\), with \(F_G\colon G\set\to \Set\) the forgetful functor, and where \(\End(F_G)\) is topologized using the compact-open topology, similarly as in Definition \ref{definition pefg}. This was shown in \cite[Proposition 4.1.1]{lepage}, and the isomorphism is given by sending \((g_U)_U\in \varprojlim_U G/U\) to the endomorphism mapping \(s\in S\in G\set\) to \(g_{G_s}\cdot s\), where \(G_s\subseteq G\) is the stabilizer of \(s\) in \(S\). Moreover, the natural map \(G\to \varprojlim_U G/U\) can be identified with the inclusion \(\Aut(F_G)\to \End(F_G)\), so that the image of this map are exactly the invertible elements of the monoid \(\End(F_G)\).
	
	\begin{lemma}\label{no partial inverses}
		Elements of \(\End(F_G)\) with a left or right inverse are already invertible.
	\end{lemma}
	\begin{proof}
		Let \(\phi,\psi\in \End(F_G)\) be such that \(\phi\circ \psi = \operatorname{Id}_{F_G}\), and consider any \(S\in G\set\).
		Since the map \(\phi_S\colon S\to S\) is surjective, it will be enough to show \(\phi_S\) is injective. Indeed, this will imply each \(\phi_S\), and hence each \(\psi_S = \phi_S^{-1}\), is bijective.
		
		Let \((\alpha_U)_U\in \varprojlim_U G/U\) be the element corresponding to \(\phi\). Choosing for each open subgroup \(U\subseteq G\) a representative \(g_U\in \alpha_U\subseteq G\) gives a net in \(\End(F_G)\), which converges to \(\phi\). Now, there is a natural map \[\Phi\colon \End(F_G)\to \map(S,S).\] If we equip the latter with the compact-open topology, this map is continuous by definition of the topology on \(\End(F_G)\). In particular, we have a net \((\Phi(g_U))_U\) in \(\map(S,S)\), which converges to \(\phi_S\). As a limit of injective (even bijective) maps into a discrete set, we conclude that \(\phi_S\) is itself injective.
	\end{proof}

	The following theorem classifies pro-étale \(\Ff_G\)-torsors, for Noohi groups \(G\). Note that the case of profinite groups already appeared in \cite[Lemma 7.4.3]{proetale}, and one can use an argument similar to the proof of \cite[Lemma 7.4.7]{proetale} to show the claim for locally profinite groups. Our method, on the other hand, works for all Noohi groups, and does not need to reduce to some simpler case. For two topological groups \(G\) and \(H\), let us denote by \(\underline{\Hom}_{\cont}(H,G)\) the groupoid of continuous homomorphisms \(H\to G\), where the morphisms between \(f_1,f_2\colon H\to G\) are given by elements \(g\in G\) conjugating \(f_1\) into \(f_2\).
	
	\begin{theorem}\label{pefg classifies torsors}
		If \(G\) is a Noohi group, there is an equivalence \[B\Ff_G(X_\pet)\cong \underline{\Hom}_{\cont}(\pefg(X,x),G).\]
	\end{theorem}
	\begin{proof}
		By \cite[Theorem 7.2.5 (2)]{proetale}, \(\underline{\Hom}_{\cont}(\pefg(X,x),G)\) is equivalent to the groupoid of functors \(G\set\to \Loc_X\) compatible with the fibre functors \(F_G\) and \(\ev_x\). So it is enough to show that \(B\Ff_G(X_\pet)\) is equivalent with this groupoid of functors.
		
		For a pro-étale \(\Ff_G\)-torsor \(\Pp\) on \(X\), we get a functor \(G\set\to \Loc_X\) by sending \(S\in G\set\) to the contracted product \(\Pp\times^{\Ff_G} \Ff_S\), where the action of \(\Ff_G\) on \(\Ff_S\) is the one induced by the action of \(G\) on \(S\). To show compatibility with the fibre functors, let us fix \(q\in \Pp_x\), which can be identified with the underlying set of \(G\). Then there is a natural isomorphism \(S\xrightarrow{\cong} \ev_x(\Pp\times^{\Ff_G} \Ff_S) \colon s\mapsto (q,s)\). Note that although we had to choose some \(q\in \Pp_x\) to show this compatibility, the functor \(G\set\to \Loc_X\) itself is independent of this choice.
		
		On the other hand, suppose we have a functor \(\Phi\colon G\set\to \Loc_X\), compatible with the fibre functors. For each \(S\in G\set\), the action of \(\End(F_G)\) on \(\ev_x(\Phi(S))=S\) extends uniquely to an action of \(\Ff_{\End(F_G)}\) on \(\Phi(S)\). Define a sheaf \[\Tilde{\Pp}\in \Shv(X_\pet) \quad \text{as } \varprojlim_U \Phi(G/U),\] where \(U\) ranges over the open subgroups of \(G\). Then the \(\Ff_{\End(F_G)}\)-action on each \(\Phi(G/U)\) lifts to a unique \(\Ff_{\End(F_G)}\)-action on \(\Tilde{\Pp}\). To construct a torsor, we may assume \(X\) is affine, as the general case will follow by gluing. In that case, \(X\) admits an affine w-contractible pro-étale cover \(W\) by \cite[Theorem 1.5]{proetale}. Such w-contractible schemes are defined by the property that any pro-étale cover over them splits. In particular, any locally constant pro-étale sheaf over \(W\) is already constant, so that for any \(T\in W_\pet\), we have an \(\End(F_G)\)-equivariant isomorphism
		\begin{equation}\label{sections of limit}
			\Tilde{\Pp}(T)=\varprojlim_U \Phi(G/U)(T)\cong\varprojlim_U \cont(T,G/U) = \cont(T,\varprojlim_U G/U).
		\end{equation}
		In particular, \(\Tilde{\Pp}(T)\cong \varprojlim_U G/U\) if \(T\) is connected. Now observe that by Lemma \ref{no partial inverses}, the subset \(\Aut(F_G)\subseteq \End(F_G)\) can be characterized as those elements \(\phi\in \End(F_G)\) for which \(\End(F_G)\cdot \phi=\End(F_G)\). In particular, for connected \(T\in W_\pet\), the \(\End(F_G)\)-action allows us to recover which elements of \(\Tilde{\Pp}(T)\) must correspond to elements of \(G\), and for connected \(T'\to T\in W_\pet\), the restriction \(\Tilde{\Pp}(T)\to \Tilde{\Pp}(T')\) preserves these elements. This allows us to define the subpresheaf \(\Pp\) of \(\Tilde{\Pp}\) whose sections over some \(Y\in X_\pet\) consist of those sections whose restriction to \(T\) lies in \(G\subseteq \Tilde{\Pp}(T)\), for any connected \(T\to Y\) with \(T\in W_\pet\). Then \(\Pp\) is a sheaf, the \(\Ff_{\End(F_G)}\)-action on \(\Tilde{\Pp}\) restricts to an \(\Ff_G\)-action on \(\Pp\), and for any \(T\in W_\pet\), we have an equivariant isomorphism \(\Pp(T)\cong \cont(T,G)\). In particular, \(\Pp\) becomes isomorphic to \(\Ff_G\) over \(W\), which shows that \(\Pp\) is an \(\Ff_G\)-torsor.
		
		These mappings can be upgraded to functors, which are readily seen to be mutual quasi-inverses.
	\end{proof}
	
	In particular, taking isomorphism classes of these groupoids gives Theorem \ref{theorem 3}.

	\section{The Siegel threefold}\label{section siegel}
	
	In this final section, we use our previous results to deduce the main theorem. Let us fix a prime \(p\), some integer \(N\geq 3\) prime to \(p\), and consider the moduli space \(\Aa_{2,1,N}\) of \emph{principally polarized abelian surfaces with level-\(N\) structure} over \(\overline{\IF_p}\). As it is three-dimensional and corresponds to the Siegel Shimura datum, it is also known as the \emph{Siegel threefold}. The associated basic Newton stratum \(V_N\) is exactly the locus of supersingular abelian surfaces, which is one-dimensional, cf.~\cite[p.217]{rapoportnewton}. To study this supersingular locus, we will construct families of principally polarized supersingular abelian surfaces over \(\IP^1_{\overline{\IF_p}}\). These families will then give rise to maps \(\IP^1_{\overline{\IF_p}}\to V_N\), which will allow us to get a grasp on the geometry of \(V_N\). This construction is originally due to Moret-Bailly, cf.~\cite{mbp}, and works for all \(p\geq 3\). So we assume that \(p\geq 3\), and refer to \cite{mb2} for a similar construction when \(p=2\). For the rest of this section, we let \(S=\IP^1_{\overline{\IF_p}}\), and any fibre product without subscript will be over \(\spec \overline{\IF_p}\).
	
	\begin{construction}\label{construction of families of supersingular AVs}
		Choose two supersingular elliptic curves \(E_1\) and \(E_2\) over \(\overline{\IF_p}\). They both admit a natural subgroup isomorphic to \(\alpha_p\), given by the kernel of their respective Frobenius maps; let us fix such inclusions \(\alpha_p=\spec \overline{\IF_p}[\beta_i]/(\beta_i^p)\into E_i\). Consider the subgroup scheme \(H = V(Y\beta_1-X\beta_2)\subset \alpha_p \times \alpha_p \times S\), where \((X,Y)\) is a homogeneous coordinate of \(S=\IP^1_{\overline{\IF_p}}\). Now we consider the quotient \(\XX:= (E_1\times E_2 \times S)/H\), giving a diagram
		\[\begin{tikzcd}
			1 \arrow[r] & H \arrow[r,"\Delta"] & E_1 \times E_2 \times S \arrow[d, "\pr_1"] \arrow[rd, "\pr_2"] \arrow[r, "\pi"] & \XX \arrow[r] \arrow[d,"q"]&1\\
			&&E_1 \times E_2 &S&
		\end{tikzcd}\] where the top row is exact. Clearly, \(\XX\) is a family of supersingular abelian surfaces over \(S\), and we want to endow it with a principal polarization. For this, consider any ample line bundle \(\Ll\) on \(E_1\times E_2\) which is symmetric (i.e., \(i^*\Ll\cong \Ll\), with \(i\colon E_1\times E_2\to E_1\times E_2\) the inversion map) and satisfies \(K_{\Ll}\cong \alpha_p\times \alpha_p\), where \(K_\Ll:= \ker(A\to A^\vee\colon  x\mapsto t_x^*\Ll)\) and \(t_x\) is the translation by an element \(x\) of \(A\). Then, as \(H\subseteq K_{\pr_1^*(\Ll)}\) and \(H\) is totally isotropic for the commutator pairing \(e^{\pr_1^*(\Ll)}\colon K_{\pr_1^*(\Ll)} \times K_{\pr_1^*(\Ll)} \to \IG_{m,S}\) associated to \(\pr_1^*(\Ll)\) (since this pairing is alternating, the fibres of \(H\) are isomorphic to \(\alpha_p\), and there are no non-zero homomorphisms \(\alpha_p\to \IG_m\)), we see that \(\pr_1^*(\Ll)\) descends to a unique line bundle \(\Mm\) on \(\XX\), i.e., \(\pi^*(\Mm)\cong \pr_1^*(\Ll)\). Moreover, as \(E_1\times E_2\times S\to \XX\) is an isogeny and \(\Mm\) pulls back to an \(S\)-relatively ample line bundle, \(\Mm\) is \(S\)-relatively ample itself. And since both the polarization induced by \(\pr_1^*(\Ll)\), and the isogeny \(E_1\times E_2 \times S\to \XX\), have degree \(p\), it follows that \(\Mm\) induces a principal polarization on \(\XX\).
	\end{construction}
	
	\begin{remark}\label{remark products of ss elliptic curves are independent}
		In fact, it does not matter which supersingular elliptic curves are chosen. Indeed, due to a theorem of Deligne, a proof of which can be found in \cite[Theorem 3.5]{shioda}, all products of \(g\geq 2\) supersingular elliptic curves are isomorphic. Moreover, since the inclusion \(\alpha_p \times \alpha_p\into E_1\times E_2\) is given by the kernel of the Frobenius map, this is also independent of the chosen elliptic curves.
	\end{remark}
	
	Now, consider a principally polarized family \(\XX\to S\) of supersingular abelian surfaces as constructed above. This is an abelian scheme, and since \(N\) is prime to \(p\), the \(N\)-torsion subgroup \(\XX[N]\subset \XX\) is finite étale over \(S\). But since \(S\cong \IP^1_{\overline{\IF_p}}\) is simply connected, any finite étale cover is trivial, so that \(\XX[N]\) is a disjoint union of copies of \(S\). Using this decomposition, we can easily put a level-\(N\) structure on the family \(\XX\to S\). By the moduli interpretation of \(V_N\), this then corresponds to a map \(S\to V_N\). Moreover, the morphism \(S\to V_N\) is not constant, as \(\XX\to S\) is not a constant family. Indeed, by Remark \ref{remark products of ss elliptic curves are independent}, it is enough to show that if \(E/\overline{\IF_p}\) is a supersingular elliptic curve, and \(H\subseteq E\times E\times S\) the subgroup from Construction \ref{construction of families of supersingular AVs}, then \((E\times E\times S)/H\) has fibres that are isomorphic to the product of two supersingular elliptic curves, and fibres that are not. For the former, one can just take \(s=0\in S\), and the latter is shown in \cite[Introduction]{oortabelian}. Finally, since both \(S\) and \(V_N\) are one-dimensional, we see that the image of this map is an irreducible component of \(V_N\).
	
	\begin{remark}\label{pefg of ss locus}
		It was proven in \cite[Theorems 2.1 and 2.3]{katsuraoort} that not only is every irreducible component of \(V_N\) the image of a map \(S\to V_N\) as constructed above, but also that these maps realize \(S=\IP^1_{\overline{\IF_p}}\) as the normalization of such an irreducible component. As \(V_N\) is moreover connected, cf.~\cite[Theorem 7.3 and Corollary 8.4]{oortstratification}, we can use Theorem \ref{pefg of curves over separably closed field} to conclude that the pro-étale fundamental group of \(V_N\) is a discrete free group on a finite number of generators. In fact, one can determine the exact number of generators using \cite[Theorems 2.4, 5.1, and 5.3]{katsuraoort}, but we will not need this. While we will also not use it, it might be interesting to note that the non-normal points of \(V_N\) are exactly those whose corresponding abelian surface is isomorphic to the product of two supersingular elliptic curves. Finally, what we will use later, is that all singularities of \(V_N\) are ordinary singularities. This is proven in \cite[p193]{koblitz}, and it implies that one can glue copies of \(S=\IP^1_{\overline{\IF_p}}\) along closed points to get \(V_N\) itself, not just its seminormalization.
	\end{remark}
	
	Now let us describe the pro-étale torsor we are interested in. While its precise definition is a bit involved, we will not need it in the following, so we content ourselves by giving the description below. Let us fix a geometric basepoint \[x\colon \spec \overline{\IF_p}\to V_N,\] with associated supersingular abelian surface \(A_x\), principal polarization \(\lambda_x\), and level structure \(\eta_x\). Recall that \(J_b(\IQ_p)\) is the group of self-quasi-isogenies of the \(p\)-divisible group \(A_x[p^\infty]\), respecting the induced polarization up to a scalar in \(\IQ_p^\times\). As the notation suggests, these are the \(\IQ_p\)-valued points of an algebraic group, and hence are naturally endowed with a locally profinite topology.
	
	\begin{definition}[{\cite[Proposition 4.3.13]{generic}}]
		There is a natural pro-étale \(J_b(\IQ_p)\)-torsor on \(V_N\), which, above any geometric point \(y\) of \(V_N\), parametrizes the quasi-isogenies between \(A_x[p^\infty]\) and \(A_y[p^\infty]\), respecting the polarizations up to a scalar in \(\IQ_p^\times\). We denote this torsor by \(\Pp_N\).
	\end{definition}
	
	Since \(J_b(\IQ_p)\) is locally profinite, we can use Theorem \ref{pefg classifies torsors} to get a map \[\pefg(V_N,x)\to J_b(\IQ_p),\] and it is the image of this map that we want to determine. Note that this map is only defined up to inner automorphism of \(J_b(\IQ_p)\), but such inner automorphisms only conjugate the image anyway. 
	
	To describe this map, note that it depends on a choice of element in \(\ev_x(\Pp_N)\). For this, we can simply consider the natural quasi-isogeny \[A_x[p^\infty]\to A_x[p^\infty]\] respecting the polarizations, given by the identity. Now, since all connected components of the normalization of \(V_N\) are simply connected, the proof of Theorem \ref{pefg of curves over separably closed field} allows us to describe an element of \(\pefg(V_N,x)\) by a finite sequence \((y_0,y_1,\ldots, y_{2m},y_{2m+1})\) of closed points of this normalization, such that \(y_0\) and \(y_{2m+1}\) map to \(x\) in \(V_N\), such that any two consecutive points of the form \((y_{2k},y_{2k+1})\) lie in the same connected component, and such that any two consecutive points of the form \((y_{2k-1},y_{2k})\) map to the same point of \(V_N\). Elements of \(\pefg(V_N,x)\) do not determine such sequences uniquely, but each such sequence determines a unique element of \(\pefg(V_N,x)\). 
	
	Let \(\phi\in \pefg(V_N,x)\), and choose such a sequence \((y_0,\ldots,y_{2m+1})\) associated to \(\phi\). For each pair \((y_{2k-1},y_{2k})\), let \(A_k\) be the supersingular abelian surface corresponding to their image in \(V_N\). On the other hand, choose for each pair \((y_{2k},y_{2k+1})\) a product of supersingular elliptic curves \(E_k,E'_k\) that, along with some line bundle \(\Ll\) as in Construction \ref{construction of families of supersingular AVs} and the right level structure, gives a family of principally polarized abelian surfaces over \(S\), corresponding to the connected component of \(y_{2k}\) and \(y_{2k+1}\). Then each \(A_k\) is in an obvious way a quotient of \(E_{k-1}\times E'_{k-1}\) and \(E_{k}\times E'_{k}\), both by a subgroup isomorphic to \(\alpha_p\). We get the following sequence of maps:
	\begin{equation}\label{quasi-isogeny obtained from path}
		\begin{tikzcd}[row sep=tiny]
			& E_0\times E_0'\arrow[dr] \arrow[dl] && \ldots \arrow[dr] \arrow[dl] && E_m\times E_m' \arrow[dr] \arrow[dl]&\\
			A_x && A_1 && A_{m} && A_x.
		\end{tikzcd}
	\end{equation}
	Their composition in the category of abelian varieties up to isogeny gives a quasi-isogeny \(A_x\to A_x\), which only depends on \(\phi\), not on the chosen sequence, and is also independent of the choices of supersingular elliptic curves \(E_i\) and \(E'_i\). Moreover, by the restrictions on the \(y_i\), this quasi-isogeny is compatible with both the polarization and the level structure. Finally, it induces a quasi-isogeny \(A_x[p^\infty]\to A_x[p^\infty]\) of p-divisible groups, and hence an element of \(J_b(\IQ_p)\), which is exactly the image of \(\phi\).
	
	We can say a few things about the quasi-isogeny \(A_x\to A_x\) above:
	\begin{itemize}
		\item It is of degree 1,
		\item It is a \(p\)-power quasi-isogeny, and
		\item It is compatible with both the polarization \(\lambda_x\) and the level structure \(\eta_x\).
	\end{itemize}
	
	Our goal will be to show that any quasi-isogeny satisfying these conditions will arise by a sequence as in \eqref{quasi-isogeny obtained from path}, as this gives us a description of the image of \(\pefg(V_N,x)\to J_b(\IQ_p)\). In fact, only the last condition is really necessary: any quasi-isogeny preserving the principal polarization must have degree 1, and every prime-to-\(p\) quasi-isogeny that preserves the polarization and level structure is already trivial.
	
	\begin{proposition}\label{image of the map}
		The image of the map \(\pefg(V_N,x)\to J_b(\IQ_p)\) consists exactly of those elements of \(J_b(\IQ_p)\) that are induced by self-quasi-isogenies of \(A_x\), preserving the polarization \(\lambda_x\) (not just up to a scalar) and level structure \(\eta_x\).
	\end{proposition}
	\begin{proof}
		We showed above that any element of \(\pefg(V_N,x)\) maps to such a quasi-isogeny. Now, fix a quasi-isogeny \(\psi\colon A_x\to A_x\) respecting the polarization and level structure. We want to factor \(\psi\) into a sequence of isogenies similar to \eqref{quasi-isogeny obtained from path}. Since it is a \(p\)-power quasi-isogeny, there is some \(n\geq 0\) such that \(p^n\psi\) is an isogeny \(A_x\to A_x\). So we can write \(\psi\) as the composition 
		\[A_x \xleftarrow{[p^n]} A_x \xrightarrow{p^n\psi} A_x,\]
		where the degrees of \([p^n]\) and \(p^n\psi\) agree. Moreover, since these degrees are powers of \(p\), we can inductively quotient out subgroups of their kernel of order \(p\), to factor \(\psi\) into a sequence of isogenies as follows, all of which have degree \(p\):
		\[\begin{tikzcd}[row sep=tiny]
			&&&A_x \arrow[rd] \arrow[ld]&&&\\
			&&A_k \arrow[ld, dotted, no head]&&A'_k \arrow[rd, dotted, no head]&&\\
			&A_1 \arrow[ld]&&&&A'_1 \arrow[rd]&\\
			A_x&&&&&&A_x.
		\end{tikzcd}\]
		Note that since all abelian surfaces are supersingular and we are working over \(\overline{\IF_p}\), all the kernels of these isogenies are isomorphic to \(\alpha_p\). Now, we want to obtain a factorization of \(\psi\) into a sequence of the following kind, where all the arrows are isogenies of degree \(p\):
		\begin{equation}\label{sequence of isogenies of degree p}
			\begin{tikzcd}[row sep=tiny]
				&A'_1\arrow[rd] \arrow[ld]&&\ldots \arrow[ld, dotted] \arrow[rd, dotted]&&A'_{l+1}\arrow[rd] \arrow[ld]& \\
				A_x&&A_1&&A_l&&A_x.
			\end{tikzcd}
		\end{equation}
		Again, we do this inductively, by rearranging subsequences of the form \(A_i \leftarrow A' \rightarrow A_j\). If the kernels of the two maps agree, then the composition \(A_i\leftarrow A'\rightarrow A_j\) is an isomorphism, and we can replace it by a single abelian surface. If the two kernels do not agree, they are still both isomorphic to \(\alpha_p\). So we can quotient out \(A_i\) by the image of the kernel of \(A'\to A_j\), and similarly take a quotient of \(A_j\). As these two quotients agree, we get a commutative diagram
		\[\begin{tikzcd}[row sep=tiny]
			&A'\arrow[rd] \arrow[ld]&\\
			A_i\arrow[rd]&&A_j \arrow[ld]\\
			&A'',&
		\end{tikzcd}\]
		and we can replace \(A_i\leftarrow A'\rightarrow A_j\) by \(A_i\to A''\leftarrow A_j\).
		
		Finally, we want to show that a sequence of isogenies as \eqref{sequence of isogenies of degree p} can arise from an element in \(\pefg(V_N,x)\). For this, equip each abelian surface \(A_i\) and \(A_i'\) with the polarization and level structure induced by the polarization and level structure of \(A_x\). Consider a subsequence of the form \[A_i \xleftarrow{f_i} A' \xrightarrow{f_j} A_j,\] and let \(\lambda_i\) and \(\lambda_j\) be the induced principal polarizations on \(A_i\) and \(A_j\) respectively. Recall that we may assume \(\ker(f_i)\neq \ker(f_j)\). Since both \(\ker(f_i)\) and \(\ker(f_j)\) are isomorphic to \(\alpha_p\), \cite[Theorem 2]{oortabelian} implies that \(A'\) is isomorphic to the product of two supersingular elliptic curves. Note also that \(f_i^*\lambda_i=f_j^*\lambda_j\). We claim that \(\ker(f_i^*\lambda_i)\) is isomorphic to \(\alpha_p\times \alpha_p\). Indeed, this is a finite subgroup scheme of a supersingular abelian surface of degree \(p^2\), so it must be isomorphic to either \(\alpha_p\times \alpha_p\) or a non-trivial extension of \(\alpha_p\) by itself. But such a non-trivial extension has only one subgroup isomorphic to \(\alpha_p\), while \(\ker(f_i^*\lambda_i)\) has at least two: \(\ker(f_i)\) and \(\ker(f_j)\).
		
		In particular, we can apply Construction \ref{construction of families of supersingular AVs} with the product of elliptic curves \(A'\), the polarization \(f_i^*\lambda_i\) and the right level structure, to get a family of abelian surfaces over \(S=\IP^1_{\overline{\IF_p}}\) that specializes to \(A_i\) and \(A_j\) at two points: these points are determined by the subgroups \(\ker(f_i)\) and \(\ker(f_j)\) of \(\ker(f_i^*\lambda_i)=\ker(f_j^*\lambda_j)\cong \alpha_p\times \alpha_p\). So each \(A_i\) with the induced polarization and level structure corresponds to a point of \(V_N\), the \(A'_j\)'s determine by the discussion above a path in \(V_N\) connecting all these points, and our assumption that \(\psi\) preserves \(\lambda_x\) and \(\eta_x\) assures us that this path is in fact a loop. This gives us an element of \(\pefg(V_N,x)\), which maps to \(\psi\in J_b(\IQ_p)\).
	\end{proof}
	
	\begin{remark}\label{our results agree with CS}
		We have shown that the image \(\Gamma\) of the map \(\pefg(V_N,x)\to J_b(\IQ_p)\) corresponding to \(\Pp_N\) consists of the self-quasi-isogenies of \(A_x[p^\infty]\) that are induced by self-quasi-isogenies of \(A_x\), preserving the polarization \(\lambda_x\) and the level structure \(\eta_x\). This agrees with a group defined in the proof of \cite[Theorem 6.23]{rapoportzink}, which is shown to be discrete. Using \cite[Theorem 6.30]{rapoportzink}, it follows that one can obtain the formal completion of \(\Aa_{2,1,N}\) along \(V_N\) as the quotient of some Rapoport-Zink space by \(\Gamma\). Note the slight inaccuracy in \cite[Remark 1.14]{generic} however, where it is stated that \(\Gamma\) is a discrete cocompact subgroup of \(J_b(\IQ_p)\). This is not correct: all quasi-isogenies in \(\Gamma\) have degree one, so the degree map \(J_b(\IQ_p)\to \IQ_{>0}\) factors through \(J_b(\IQ_p)/\Gamma\). But since the image of this degree map is not compact, \(J_b(\IQ_p)/\Gamma\) is not compact either, so that \(\Gamma \subset J_b(\IQ_p)\) is not cocompact. What does hold however, is that \(\Gamma \subset J_b(\IQ_p)\) is cocompact \emph{modulo the center} of \(J_b(\IQ_p)\), which is also the statement given in \cite[Theorem 6.30]{rapoportzink}.
	\end{remark}
	
	Next, we show that the map \(\pefg(V_N,x)\to J_b(\IQ_p)\) is injective. In the proof, the following observation will be used multiple times: let \(G\) be a finite subgroup scheme of a supersingular abelian variety over \(\overline{\IF_p}\), of degree \(p^n\) with \(n\geq 1\). If \(G\) does not contain a subgroup scheme isomorphic to \(\alpha_p\times \alpha_p\), then \(G\) has only one subgroup scheme of degree \(p\).
	
	\begin{proposition}\label{injectivity of the map}
		The map \(\pefg(V_N,x)\to J_b(\IQ_p)\) is injective.
	\end{proposition}
	\begin{proof}
		Consider a sequence of degree \(p\) isogenies between principally polarized abelian surfaces with level structure, of the form
		\begin{equation}\label{sequence we want to show is trivial}
			\begin{tikzcd}[row sep=tiny]
				& E_0\times E_0'\arrow[dr] \arrow[dl] && \ldots \arrow[dr] \arrow[dl] && E_m\times E_m' \arrow[dr] \arrow[dl]&\\
				A_x && A_1 && A_{m} && A_x
			\end{tikzcd}
		\end{equation}
		as in \eqref{quasi-isogeny obtained from path}. Let us assume that the composition \(A_x\to A_x\) is just the identity, in which case we need to show this sequence determines the trivial element of \(\pefg(V_N,x)\). Our goal will be to find a different, strictly shorter sequence determining the same element of \(\pefg(V_N,x)\), so that the proposition will follow by induction. Consider the commutative diagram
		\begin{equation}\label{big diagram of isogenies}
			\begin{tikzcd}[row sep=tiny, column sep=tiny]
				&&&&B_{m+2,0} \arrow[ld] \arrow[rd]&&&&\\
				&&&B_{m+1,0} && B_{m+1,1}&&&\\
				&&&\iddots&\vdots&\ddots&&&\\
				&&B_{2,0} \arrow[ld] \arrow[rd]&&&& B_{2,m-1} \arrow[ld] \arrow[rd]\\
				&B_{1,0} \arrow[ld] \arrow[rd] &&B_{1,1} \arrow[ld] \arrow[rd]&&B_{1,m-1} \arrow[ld] \arrow[rd] &&B_{1,m} \arrow[ld] \arrow[rd]\\
				A_x && A_1 && \ldots && A_{m} && A_x,
			\end{tikzcd}
		\end{equation}
		where we denote \(E_i\times E_i'\) by \(B_{1,i}\), and the \(B_{j,i}\) for \(j>1\) are defined as the reductions of the obvious fibre products. We will also denote \(A_i\) by \(B_{0,i}\) and \(A_x\) by \(B_{0,0}\), \(B_{0,m+1}\), or \(A_0\), depending on the situation. One can inductively show that the \(B_{j,i}\)'s are supersingular abelian surfaces: as reducedness implies smoothness and isogenies are proper, we only need to show they are connected. But this follows from \cite[Tag 0377]{stacks} and the fact that all our isogenies have degree \(1\) or \(p\). For the rest of this proof, we will use the term \emph{arrow} to denote a single isogeny of degree 1 or \(p\) in \eqref{big diagram of isogenies} or \eqref{sequence we want to show is trivial}, and use \emph{isogeny} for compositions of such arrows.
		We want to show that there is some \(B_{j,i}\), with \(j\in \{1,2\}\), such that the two arrows out of it have the same kernel. Indeed, in that case two consecutive arrows in \eqref{sequence we want to show is trivial} will be equal, so that removing them gives a strictly shorter sequence that determines the same element of \(\pefg(V_N,x)\).
		
		Let us start by showing that there exists some \(B_{j,i}\) with \(j\geq 1\), such that the two arrows out of it have the same kernel. This is clear if there is some arrow in \eqref{big diagram of isogenies} of degree 1: then one can take some \(B_{j,i}\), with \(j\geq 1\) minimal such that one arrow out of \(B_{j,i}\) has degree 1; minimality of \(j\) then implies both arrows out of \(B_{j,i}\) have degree 1. So we may assume all arrows have degree \(p\). First, assume there is a composition of arrows \(B_{m+2,0}\to A_i\) (where \(A_x=A_0\) is allowed) whose kernel does not contain a subgroup of the form \(\alpha_p\times\alpha_p\). Then the same holds for any composition which starts at the other arrow going out of \(B_{m+2,0}\), but where the composed isogeny is the same; we can always find at least one such composition by our assumption that the composed quasi-isogeny \(A_x\to A_x\) is the identity. But then the kernels of these composed isogenies \(B_{m+2,0}\to A_i\) contain only one subgroup of order \(p\), so that the arrows out of \(B_{m+2,0}\) have the same kernel. On the other hand, if the kernels of all isogenies \(B_{m+2,0}\to A_i\) in \eqref{big diagram of isogenies} contain a subgroup \(\alpha_p\times \alpha_p\), choose a composition \(B_{m+2,0}\to B_{j,i}\), with \(j\geq 1\) minimal such that the composed isogeny does not contain a subgroup of the form \(\alpha_p\times \alpha_p\). Then both \(B_{j-1,i}\) and \(B_{j-1,i+1}\) can be obtained by quotienting out the image of \(\alpha_p\times\alpha_p\subseteq B_{m+2,0}\) in \(B_{j,i}\), and in particular, the two arrows out of \(B_{j,i}\) have the same kernel. (Here, we used that abelian surfaces have at most one subgroup scheme isomorphic to \(\alpha_p\times \alpha_p\).)
		
		Now, consider some \(B_{j,i}\) as above, so that the two arrows out of it have the same kernel. As we want to have \(j=1\) or \(j=2\), let us assume that \(j>2\); again we may assume that all arrows with source \(B_{j',i'}\) for \(j'<j\) have degree \(p\), as otherwise we could already find a smaller \(j\) with the desired property. Note that \(B_{j-1,i}\) and \(B_{j-1,i+1}\) get canonically identified. If the two compositions \(B_{j-1,i}\to B_{j-3,i+1}\) do not have kernel isomorphic to \(\alpha_p\times \alpha_p\), then the kernels of the arrows out of \(B_{j-1,i}\) agree. A similar statement holds if we replace \(B_{j-1,i}\) by \(B_{j-1,i+1}\). On the other hand, if the isogenies \(B_{j-1,i}\to B_{j-3,i+1}\) and \(B_{j-1,i+1}\to B_{j-3,i+2}\) both have kernel isomorphic to \(\alpha_p\times \alpha_p\), then both \(B_{j-3,i+1}\) and \(B_{j-3,i+2}\) can be obtained by quotienting out the image of (the unique subgroup scheme) \(\alpha_p\times \alpha_p\subseteq B_{j-1,i}=B_{j-1,i+1}\) in \(B_{j-2,i+1}\). In particular, the arrows out of \(B_{j-2,i+1}\) have the same kernel. Continuing in this fashion, we can find \(j\in \{1,2\}\) and a \(B_{j,i}\) with this property, concluding the proof.
	\end{proof}
	
	Let us end this paper with two immediate corollaries to Propositions \ref{image of the map} and \ref{injectivity of the map}:
	
	\begin{corollary}
		Let \(A\) be a supersingular abelian surface over \(\overline{\IF_p}\), equipped with a principal polarization and level-\(N\) structure. Then the group of self-quasi-isogenies of \(A\) respecting this extra structure is a free group on a finite number of generators.
	\end{corollary}
	
	Our second corollary is concerned with Rapoport-Zink spaces as defined in \cite{rapoportzink}. Since we do not use these spaces in the rest of this paper, we omit their definition, and refer to loc.~cit.~for the theory behind them. Let us just mention that they are formal schemes over \(\spf \Breve{\IZ}_p = \spf W(\overline{\IF_p})\), but by topological invariance, we can and do consider their reduced subschemes instead, which live over \(\spec \overline{\IF_p}\). Moreover, recall from Remark \ref{our results agree with CS} that by \(p\)-adic uniformization, \(V_N\) is canonically a quotient \(\Gamma\backslash \Mm_b\), where \(\Mm_b\) is the Rapoport-Zink space corresponding to \(V_N\), and \(\Gamma\) is the image of \(\pefg(V_N,x)\to J_b(\IQ_p)\), as before.
	
	\begin{corollary}
		The Rapoport-Zink space \(\Mm_b\) corresponding to the basic stratum of the Siegel threefold is simply connected.
	\end{corollary}
	\begin{proof}
		Essentially by definition of both \(\Mm_b\) and \(\Pp_N\), \(\Mm_b\) is the smallest geometric cover of \(V_N\) trivializing \(\Pp_N\). Hence, as in topology, \(\Mm_b\) being simply connected is equivalent to injectivity of the natural map \(\pefg(V_N,x)\to \Gamma\). We conclude by Proposition \ref{injectivity of the map}.
	\end{proof}


\begin{thebibliography}{{Lar}22}
	
	\bibitem[AT08]{topgrps}
	Alexander Arhangel'skii and Mikhail Tkachenko.
	\newblock {\em Topological groups and related structures}, volume~1 of {\em
		Atlantis Studies in Mathematics}.
	\newblock Atlantis Press, Paris; World Scientific Publishing Co. Pte. Ltd.,
	Hackensack, NJ, 2008.
	
	\bibitem[BS15]{proetale}
	Bhargav Bhatt and Peter Scholze.
	\newblock The pro-\'{e}tale topology for schemes.
	\newblock {\em Ast\'{e}risque}, (369):99--201, 2015.
	
	\bibitem[CS17]{generic}
	Ana Caraiani and Peter Scholze.
	\newblock On the generic part of the cohomology of compact unitary {S}himura
	varieties.
	\newblock {\em Ann. of Math. (2)}, 186(3):649--766, 2017.
	
	\bibitem[Fer03]{ferrand}
	Daniel Ferrand.
	\newblock Conducteur, descente et pincement.
	\newblock {\em Bull. Soc. Math. France}, 131(4):553--585, 2003.
	
	\bibitem[Gro65]{ega4.2}
	A.~Grothendieck.
	\newblock \'{E}l\'{e}ments de g\'{e}om\'{e}trie alg\'{e}brique. {IV}. \'{E}tude
	locale des sch\'{e}mas et des morphismes de sch\'{e}mas. {II}.
	\newblock {\em Inst. Hautes \'{E}tudes Sci. Publ. Math.}, (24):231, 1965.
	
	\bibitem[Gro67]{ega4.4}
	A.~Grothendieck.
	\newblock \'{E}l\'{e}ments de g\'{e}om\'{e}trie alg\'{e}brique. {IV}. \'{E}tude
	locale des sch\'{e}mas et des morphismes de sch\'{e}mas {IV}.
	\newblock {\em Inst. Hautes \'{E}tudes Sci. Publ. Math.}, (32):361, 1967.
	
	\bibitem[KO87]{katsuraoort}
	Toshiyuki Katsura and Frans Oort.
	\newblock Families of supersingular abelian surfaces.
	\newblock {\em Compositio Math.}, 62(2):107--167, 1987.
	
	\bibitem[Kob75]{koblitz}
	Neal Koblitz.
	\newblock {$p$}-adic variation of the zeta-function over families of varieties
	defined over finite fields.
	\newblock {\em Compositio Math.}, 31(2):119--218, 1975.
	
	\bibitem[{Lar}19]{lara}
	Marcin {Lara}.
	\newblock {Fundamental exact sequence for the pro-étale fundamental group}.
	\newblock {\em arXiv e-prints}, page arXiv:1910.14015v2, October 2019.
	\newblock To appear in \emph{Algebra Number Theory}.
	
	\bibitem[Lav18]{lavanda}
	Elena Lavanda.
	\newblock Specialization map between stratified bundles and pro-\'{e}tale
	fundamental group.
	\newblock {\em Adv. Math.}, 335:27--59, 2018.
	
	\bibitem[Lep09]{lepage}
	Emmanuel Lepage.
	\newblock {\em Géométrie anabélienne tempérée}.
	\newblock PhD thesis, 2009.
	\newblock Thèse de doctorat dirigée par André, Yves. Mathématiques Paris 7,
	2009.
	
	\bibitem[LYZ22]{lara2}
	Marcin {Lara}, Jiu-Kang {Yu}, and Lei {Zhang}.
	\newblock {A theorem on meromorphic descent and the specialization of the
		pro-{\'e}tale fundamental group}.
	\newblock {\em arXiv e-prints}, page arXiv:2103.11543v3, February 2022.
	
	\bibitem[MB79]{mb2}
	Laurent Moret-Bailly.
	\newblock Polarisations de degr\'{e} {$4$} sur les surfaces ab\'{e}liennes.
	\newblock {\em C. R. Acad. Sci. Paris S\'{e}r. A-B}, 289(16):A787--A790, 1979.
	
	\bibitem[MB81]{mbp}
	Laurent Moret-Bailly.
	\newblock Familles de courbes et de vari\'{e}t\'{e}s ab\'{e}liennes sur {${\Bbb
			P}^1$}. {II}. {E}xemples.
	\newblock {\em Ast\'{e}risque}, (86):125--140, 1981.
	\newblock Seminar on Pencils of Curves of Genus at Least Two.
	
	\bibitem[Oor75]{oortabelian}
	Frans Oort.
	\newblock Which abelian surfaces are products of elliptic curves?
	\newblock {\em Math. Ann.}, 214:35--47, 1975.
	
	\bibitem[Oor01]{oortstratification}
	Frans Oort.
	\newblock A stratification of a moduli space of abelian varieties.
	\newblock In {\em Moduli of abelian varieties ({T}exel {I}sland, 1999)}, volume
	195 of {\em Progr. Math.}, pages 345--416. Birkh\"{a}user, Basel, 2001.
	
	\bibitem[Rap03]{rapoportnewton}
	Michael Rapoport.
	\newblock On the {N}ewton stratification.
	\newblock {\em Ast\'{e}risque}, (290):Exp. No. 903, viii, 207--224, 2003.
	\newblock S\'{e}minaire Bourbaki. Vol. 2001/2002.
	
	\bibitem[RZ96]{rapoportzink}
	M.~Rapoport and Th. Zink.
	\newblock {\em Period spaces for {$p$}-divisible groups}, volume 141 of {\em
		Annals of Mathematics Studies}.
	\newblock Princeton University Press, Princeton, NJ, 1996.
	
	\bibitem[Ser80]{serretrees}
	Jean-Pierre Serre.
	\newblock {\em Trees.}
	\newblock Springer-Verlag, Berlin-New York, 1980.
	\newblock Translated from the French by John Stillwell.
	
	\bibitem[SGA1]{sga1}
	{\em Rev\^{e}tements \'{e}tales et groupe fondamental.}
	\newblock Springer-Verlag, Berlin-New York, 1971.
	\newblock S\'{e}minaire de G\'{e}om\'{e}trie Alg\'{e}brique du Bois Marie
	1960--1961 (SGA 1), Dirig\'{e} par Alexandre Grothendieck. Augment\'{e} de
	deux expos\'{e}s de M. Raynaud.
	
	\bibitem[SGA3]{sga3.2}
	{\em Sch\'{e}mas en groupes. {II}: {G}roupes de type multiplicatif, et
		structure des sch\'{e}mas en groupes g\'{e}n\'{e}raux.}
	\newblock Springer-Verlag, Berlin-New York, 1970.
	\newblock S\'{e}minaire de G\'{e}om\'{e}trie Alg\'{e}brique du Bois Marie
	1962/64 (SGA 3), Dirig\'{e} par M. Demazure et A. Grothendieck.
	
	\bibitem[Shi79]{shioda}
	Tetsuji Shioda.
	\newblock Supersingular {$K3$} surfaces.
	\newblock In {\em Algebraic geometry ({P}roc. {S}ummer {M}eeting, {U}niv.
		{C}openhagen, {C}openhagen, 1978)}, volume 732 of {\em Lecture Notes in
		Math.}, pages 564--591. Springer, Berlin, 1979.
	
	\bibitem[{Sta}]{stacks}
	The {Stacks Project Authors}.
	\newblock {\em \itshape Stacks Project}.
	
	\bibitem[Sti06]{stix}
	Jakob Stix.
	\newblock A general {S}eifert-{V}an {K}ampen theorem for algebraic fundamental
	groups.
	\newblock {\em Publ. Res. Inst. Math. Sci.}, 42(3):763--786, 2006.
	
\end{thebibliography}
\end{document}